\documentclass[reqno]{amsart}
\usepackage{amsmath,amssymb,amsfonts,bbm,stmaryrd,hhline}
\usepackage[arrow,matrix,curve]{xy}  
\usepackage{latexsym}
\usepackage{graphics}

\usepackage{epsfig}
\usepackage{color}
\usepackage{amsthm}
\usepackage{amssymb}
\usepackage{amsmath}
\usepackage{amscd}
\usepackage{amsopn}

\usepackage{url}
\usepackage{hyperref}\hypersetup{colorlinks}

\usepackage{color} 

\definecolor{darkred}{rgb}{0.8,0,0} 
\definecolor{darkgreen}{rgb}{0,0.6,0}
\definecolor{darkblue}{rgb}{0,0,0.8}

\hypersetup{colorlinks,
linkcolor=darkblue,
filecolor=darkgreen,
urlcolor=darkred,
citecolor=darkgreen}

\newtheorem{thm}{Theorem}[section]

\newtheorem{prop}[thm]{Proposition}

\theoremstyle{definition}
\newtheorem{definition}[thm]{Definition}
\theoremstyle{remark}
\newtheorem{rmk}[thm]{Remark}
\newtheorem{example}[thm]{Example}

\newcommand{\eps}{\epsilon}

\newcommand{\R}{\mathbb{R}}

\newcommand{\Z}{\mathbb{Z}}

\def    \CF     {\operatorname{CF}}
\def    \HF     {\operatorname{HF}}
\def    \SH     {\operatorname{SH}}
\newcommand{\CZ}{\text{CZ}}
\newcommand{\SDMax}{symplectically degenerate maximum}
\newcommand{\SDMaxs}{symplectically degenerate maxima}
\newcommand{\SDMin}{symplectically degenerate minimum}
\newcommand{\SDMins}{symplectically degenerate minima}

\numberwithin{equation}{section}

\begin{document}

\title{The Conley conjecture for the cotangent bundle}
\author{Doris Hein}
\address{Department of Mathematics, UC Santa Cruz,
Santa Cruz, CA 95064, USA}
\email{dhein@ucsc.edu}

\date{\today}

\begin{abstract}
We prove the Conley conjecture for cotangent bundles of oriented, closed manifolds, and Hamiltonians, which are quadratic at infinity, i.e., we show that such Hamiltonians have infinitely many periodic orbits. 
For the conservative systems, similar results  have been proven by Lu and Mazzucchelli using convex Hamiltonians and Lagrangian methods.
Our proof uses Floer homological methods from Ginzburg's proof of the Conley conjecture for closed symplectically aspherical manifolds.
\end{abstract}
\maketitle

\tableofcontents
\section{Introduction}
\label{sec:intro}

\subsection{Main results}
\label{sec:Conley}

In the present paper, we establish a version of the Hamiltonian Conley conjecture for cotangent bundles.
Let $B$ be a closed, oriented manifold and let $M=T^\ast B$ be the cotangent bundle of $B$ with the canonical symplectic structure $\omega$. 
For cotangent bundles, a natural class of Hamiltonians to work with, from both technical and conceptual point of view, is that of Hamiltonians quadratic at infinity; cf. \cite{AS}.
More concretely, let $H\colon S^1\times M\to \R$ be a Hamiltonian on $M$, satisfying the following conditions:

\begin{enumerate}
\item [\textbf{(H1)}] there exist constants $h_0>0$ and $h_1\geq 0$, such that 
\[dH(t,q,p)\left(p\frac{\partial}{\partial p}\right)-H(t,q,p)\geq h_0\|p\|^2-h_1\text{ and }\]
\item [\textbf{(H2)}] there exists a constant $h_2\geq 0$, such that 
\[\|\nabla_qH(t,q,p)\|\leq h_2(1+\|p\|^2) \text{ and } \|\nabla_pH(t,q,p)\|\leq h_2(1+\|p\|).\]
\end{enumerate}

In this setting, the main result of this paper is the following theorem:

\begin{thm}[The Conley Conjecture]
Let $M$ and $H$ be as above and let $\varphi$ be the time-1-map of the Hamiltonian flow of $H$. Assume that $\varphi$ has only finitely many fixed points.
Then $\varphi$ has simple periodic orbits of arbitrarily large period.
\label{ConlConj}
\end{thm}

This is the cotangent bundle version of a conjecture Conley stated in 1984 in \cite{Co} for the case that  $M=T^{2n}$ is a torus.
It has been proven for weakly non-degenerate Hamiltonian diffeomorphisms of tori in \cite{CZ} and of symplectically aspherical manifolds in \cite{SZ}. In \cite{FrHa}, the conjecture was proven for all Hamiltonian diffeomorphisms of surfaces other than $S^2$. In its original form, as stated 
in \cite{Co} for all Hamiltonian diffeomorphisms of tori, the conjecture was established in \cite{Hi} and the case of an arbitrary closed, symplectically aspherical manifold was settled in \cite{Gi}. This proof was extended to closed symplectically rational manifolds $M$ with $c_1(M)|_{\pi_2(M)}=0$ in~\cite{GG-gaps} and to general closed symplectic manifolds with  $c_1(M)|_{\pi_2(M)}=0$ in~\cite{He}.

The papers of Hingston and Ginzburg provided the tool of symplectically degenerate extrema also utilized here for dealing with the degenerate case of the Conley conjecture for compact manifolds; see \cite{Gi,GG-gaps,He,Hi}. Our argument also uses recent results on the Floer homology of Hamiltonians on cotangent bundles proven e.g. in~\cite{AS,Se,Vi}.

The Lagrangian version of the conjecture has been considered, e.g., in \cite{Lo-CC,Lu,Ma}. The similarity between the problems can also easily be seen on the level of the proofs, although the methods utilized in \cite{Lo-CC,Lu,Ma} are quite different from the Floer homological techniques used in the present paper and in the proofs for the case of closed manifolds.
The class of Hamiltonians quadratic at infinity includes the Hamiltonians of classical mechanics (see Example \ref{classical ex}) and convex quadratic Hamiltonians used in \cite{Lo-CC,Lu}, but does not include all Tonelli Hamiltonians. 
Namely, Tonelli Hamiltonians are assumed to satisfy a convexity condition and to have superlinear (but not necessarily quadratic) growth, while Hamiltonians quadratic at infinity only need to have quadratic growth outside a compact set.

\begin{rmk}
The choice of a class of Hamiltonians in these questions is important from at least a technical points of view. The conditions on the Hamiltonian assumed here are used in the proof of Theorem~\ref{ConlConj} only to ensure that the Floer homology is well-defined, independent of the choice of the Hamiltonian $H$ and that $\HF_{-n}(H)\neq 0$. 
Under the conditions (H1) and (H2), these properties have been proven in \cite{AS}; see also \cite{Se,SW,Vi}. 
One can expect a result similar to Theorem~\ref{ConlConj} to hold for other classes of Hamiltonians. For example, our argument also goes through for  Hamiltonians which, at infinity, are autonoumous and fiberwise convex and have superlinear growth; see Remark \ref{rmk:Liouville}.
For the classes of convex and quadratic Hamiltonians or Tonelli Hamiltonians, the theorem was proven by means of Lagrangian methods in \cite{Lo-CC,Lu,Ma}.
It is not clear if the Floer homology is defined for Tonelli Hamiltonians, since the Hamiltonian flow is not automatically complete.
Note, however, that the Floer homology is defined if the Hamiltonian is autonomous and Tonelli at infinity. But even then, the invariance of filtered Floer homology has not been established, cf.~\cite{Se,Vi}.
\end{rmk}

\begin{rmk}
\label{rmk:Liouville}
Our proof of Theorem \ref{ConlConj} can also be used in a more general setting. Namely, let $(M,\omega)$ be a Liouville domain and let $H$ be a Hamiltonian, which is autonomous and depends only on the "radial variable" at infinity and has superlinear growth. The Floer homology $\HF(H)$ of such a Hamiltonian is isomorphic to the symplectic homology $\SH(M)$; see \cite{Se}. This homology is a unital algebra with unit in degree $-n$ in our degree conventions. (Strictly speaking, the cohomology $(\HF(H))^\ast$ is such an algebra.) Hence, $\HF_{-n}(H)=\SH_{-n}(M)\neq 0$ if and only if $\SH(M)\neq 0$.
In this case, our proof of Theorem~\ref{ConlConj} goes through  and thus the Conley conjecture holds.
Strictly speaking, this is not a generalization of Theorem \ref{ConlConj} since on cotangent bundles most Hamiltonians which are quadratic at infinity are not in the above class.
\end{rmk}

By a similar argument as in \cite{GG-gaps, He} for compact manifolds, it suffices to prove Theorem \ref{ConlConj} in the presence of a \SDMin.

\begin{definition}
\label{SDMin}
An isolated $k$-periodic orbit $x$ of a $k$-periodic Hamiltonian $H$ is called a
 \textit{\SDMin} of $H$ if \[\Delta_H(x)=0 \text{ and } \HF_{-n}(H,x)\neq 0.\]
\end{definition}

In this definition, let $\Delta_H(x)$ be the mean index of $x$ and we denote by $\HF_\ast(H,x)$ the local Floer homology of $H$ at $x$.
We refer to~\cite{GG-loc} for details on the local Floer homology and to \cite{SZ} for the definition of the mean index. See also Section~\ref{sec:OrbitsandHF} for more details on the mean index and Section \ref{sec:locHF} for the definition of the local Floer homology.

Definition~\ref{SDMin} is analogous to the notion of a \SDMax\ used in \cite{Gi,GG-gaps,He,Hi}. 
The latter was first utilized in \cite{Hi} when the concept of \SDMaxs\ was introduced. It was explicitly formulated and further investigated in \cite{Gi,GG-loc}.
As we will show, \SDMins\ and \SDMaxs\ have very similar geometric properties; see Proposition~\ref{SDM geo} below.

\begin{thm}[Degenerate Conley Conjecture]
Let $M=T^\ast B$ be the cotangent bundle of a closed, oriented manifold $B$ 
 and let $H$ be a Hamiltonian on $M$ satisfying conditions (H1) and (H2). Assume furthermore that $\varphi_H$ has only finitely many fixed points, and  that $H$ has a \SDMin.
Then the Hamiltonian diffeomorphism $\varphi_H$ generated by $H$ has simple periodic orbits of arbitrarily large period.
\label{degConlConj}
\end{thm}

As in \cite{GG-gaps,He}, this theorem implies the Conley conjecture as stated in Theorem~\ref{ConlConj}. 
\begin{proof}[Proof of Theorem \ref{ConlConj}]
If there is no symplecticaly degenerate minimum, all one-periodic orbits with non-zero local Floer homology have non-zero mean index. As the mean index $\Delta_H$ grows linearly with iteration and is related to the Conley-Zehnder index grading the Floer homology, the support of local Floer homology shifts away from the interval $\left[-2n,0\right]$, i.e., the local Floer homology in those degrees eventually becomes zero. It follows then by a standard argument as in~\cite{Gi,GG-gaps,Hi,SZ} that for the $k$th iteration of $H$ we have $0= \HF_{-n}(H^{(k)})$.  On the other hand, as is shown in \cite{AS,SW,Vi}, we have $\HF_{-n}(H^{(k)})\cong H_n(\Lambda_0B)\neq 0$ in our conventions.
 This contradiction proves the theorem in the absence of a \SDMin.
If there is a \SDMin, the theorem follows from Theorem~\ref{degConlConj}.
\end{proof}

\begin{rmk}
To prove Theorem \ref{ConlConj} in the case of a closed manifold $M$, one could use a \SDMin\ or a \SDMax, since for a closed symplectic manifold $M$, we know $\HF_{n}(H)\cong H_{2n}(M)\neq 0$ and also $\HF_{-n}(H)\cong H_{0}(M)\neq 0$. The former is used in \cite{Gi,GG-gaps,He} to prove the existence of a \SDMax\ by an analog argument to the above.
In the case of the cotangent bundle $M=T^\ast B$,  the transition to a \SDMin\ is necessary to prove Theorem \ref{ConlConj}. The Floer homology for a Hamiltonian on the cotangent bundle is isomorphic to the homology of the free loop space $\Lambda_0B$ of the base manifold up to a sign change in degree. A \SDMax\ would have degree $n>0$ in the Floer chain groups and would therefore necessarily be zero in homology, since $H_{-n}(\Lambda_0B)=0$. Thus we cannot show the existence of an \SDMax\ by an argument as the above.
The existence of a \SDMin, which has negative degree in the Floer homology, follows from the isomorphism $\HF_{-n}(H)=H_{n}(\Lambda_0 B)\neq 0$ proven in \cite{AS}. Therefore we can use a \SDMin\ to show Theorem \ref{ConlConj} for the cotangent bundle.
\end{rmk}

The proof of Theorem \ref{degConlConj} is based on a Floer theoretical argument establishing

\begin{thm}
\label{theorem}
Let the manifold $M=T^\ast B$ and the Hamiltonian $H$ be as above.
Assume that $H$ has a \SDMin\ at $x$ with
$\mathcal{A}_H(x)=c$. 
Then  for every sufficiently small $\epsilon >0$ there exists some $k_{\epsilon}$ such that
\[
\HF_{-n-1}^{(kc-\eps,\,kc-\delta_k)}(H^{(k)})\neq 0 \text{  for all } k>k_{\epsilon} \text{ and some } \delta_k\in(0,\epsilon).
\]
\end{thm}

\noindent Here $H^{(k)}$ denotes the one-periodic Hamiltonian $H$ viewed as $k$-periodic function for some integer $k$, see also Section~\ref{sec:Hamflow}. This theorem implies Theorem~\ref{degConlConj} and thus also Theorem~\ref{ConlConj}.

\begin{proof}[Proof of Theorem \ref{degConlConj}]
Arguing by contradiction, we assume that for every sufficiently large period all periodic orbits are iterated. Let $k$ be a sufficiently large prime. Then every $k$-periodic orbit is an iterated one-periodic orbit. By Theorem~\ref{theorem}, there exists a $k$-periodic orbit $y_k$ with
\[
-2n-1 \leq \Delta_{H^{(k)}}(y_k) \leq -1
\]
in the Floer chain group in degree $-n-1$.
The mean index of an iterated orbit growths linearly with iteration. Thus for a one-periodic orbit, the mean indices of the iterations are zero whenever the one-periodic orbit has mean index zero. If a one-periodic orbit has non-zero mean index, then the mean index of an iteration is outside the interval $\left[-2n-1,\,-1\right]$ for a sufficiently large order of iteration. The assumption that there are only finitely many one-periodic orbits implies that for sufficiently large $k$, no $k$th iteration of a one-periodic orbit has mean index in $\left[-2n-1,\,-1\right]$. Thus the orbit $y_k$ cannot be an iterated one-periodic orbit in contradiction to the choice of $k$.
\end{proof}

\subsection{Organization of the paper}
\label{sec:orga}
In Section \ref{sec:prelims}, we set the notation and conventions used in this paper and discuss some properties of the Floer homology.
Theorem \ref{theorem} is proven in Section \ref{sec:Beweis}, starting with a geometric description of symplectically degenerate extrema in Section~\ref{sec:sdmin}.

\subsection{Acknowledgments}
The author is grateful to Viktor Ginzburg for showing her this problem, helping with the solution and his kind, thoughtful advice.
Furthermore, the auther is very grateful to Alberto Abbondandolo, Ba\c sak G\"urel and Alexandru Oancea for valuable comments.

\section{Preliminaries}
\label{sec:prelims}

In this section we will introduce the notation used in this paper and review some of the basic facts needed in order to prove the theorems.

Let $B$ be a closed, oriented manifold and let $M=T^\ast B$ be the cotangent bundle of $B$, equipped with the canonical symplectic structure $\omega$.

\subsection{Hamiltonian flows}
\label{sec:Hamflow}

All considered Hamiltonians $H$ on $M=T^\ast B$ are assumed to be one-periodic in time, i.e., functions $H\colon S^1\times M\to \mathbb{R}$ whith $S^1=\mathbb{R}/\mathbb{Z}$, and we will set $H_t(x)=H(t,\,x)$.
A one-periodic Hamiltonian $H$ can also be viewed as $k$-periodic for any integer $k$. For our argument, it is sometimes crucial to keep track of the period we are considering. If a one-periodic Hamiltonian $H$ is viewed as $k$-periodic, we refer to it as the \textit{$k$th iteration} of $H$ and denote it by $H^{(k)}$. 
In particular, the function $H^{(k)}$ is a function $H^{(k)}\colon \R/k\Z\times M\to\R$.

Throughout the paper, we assume all Hamiltonians to satisfy conditions (H1) and (H2).
These conditions imply that $H$ grows quadratically at infinity, i.e.,
\begin{equation}
H(t,q,p)\geq \frac{1}{2}h_0\|p\|^2-h_3
\label{quadgrow}
\end{equation}
for some suitable constant $h_3$.

\begin{example}
\label{classical ex}
In particular, the above growth conditions on the Hamiltonian hold for all conservative Hamiltonians describing systems from classical mechanics on $B$, i.e., Hamiltonians of the form $H(t,p,q)=\frac{1}{2}\|p\|^2+V(q)$. More generally, one could also use Hamiltonians of the form $H(t,p,q)=\frac{1}{2}\|p\|_t^2+V(t,p,q)$, where in every fiber the function $V$ is constant outside a compact set in $M=T^\ast B$ and the metric $\|\cdot\|_t$ can be chosen to be time-dependent with both $V$ and the metric are periodic in time.
The conditions are also satisfied for periodic in time electro-magnetic Hamiltonians, i.e., the Hamiltonians describing the motion of a charge in an exact magnetic field and a conservative force field; see, e.g., \cite{Gi-magn}.  For these Hamiltonians, we assume the magnetic field, the metric and the potential to be periodic in time, since we only work with time-periodic Hamiltonians here.
\end{example}

As the symplectic form $\omega$ is non-degenerate,  the Hamiltonian vector field $X_H$ of $H$ is well-defined by the equation $i_{X_H}\omega=-dH$. The time-1-map of the flow of $X_H$ is called a Hamiltonian diffeomorphism and denoted by~$\varphi_H$.

The composition $\varphi_H^t\circ\varphi_K^t$ of two Hamiltonian flows is again the flow of a Hamiltonian vector field. It is generated by the Hamiltonian
\[
(K\# H)_t=K_t+H_t\circ \varphi_K^{-t}.
\]
In general, this function is not necessarily one-periodic, even if both $H$ and $K$ are one-periodic Hamiltonians. But $K\# H$ will be one-periodic if both $H$ and $K$ are one-periodic and in addition $K$ generates a loop of Hamiltonian diffeomorphisms. This will always be the case in this paper.

\subsection{Periodic orbits and Floer homology}
\label{sec:OrbitsandHF}
The Hamiltonian action is defined by the functional
\[
\mathcal{A}_H(x)=-\int_{D^2}u^\ast\omega +\int_{S^1}H_t(x(t))\ dt
\]
on the space of contractible closed loops on $M$. Here $u\colon D^2\to M$ is a continuous map with $u|_{S^1}=x$.  Since the symplectic form $\omega$ is exact, the action $\mathcal{A}_H(x)$ is independent of the choice of the capping disk $u$. In this paper, we only work with contractible periodic orbits and every periodic orbit is assumed to be contractible, even if this is not explicitly stated.

A one-periodic orbit $x$ of $H$ is said to be \textit{non-degenerate} if the linearized return map $d\varphi_H\colon T_{x(0)}M\to T_{x(0)}M$ does not have one as an eigenvalue. Following \cite{SZ}, we call a degenerate orbit \textit{weakly non-degenerate} if at least one eigenvalue is not equal to one and \textit{strongly degenerate} otherwise. We refer to a Hamiltonian $H$ as \textit{non-degenerate}, if all its one-periodic orbits are non-degenerate.

Let $\Delta_H(x)$ denote the \textit{mean index} of a one-periodic orbit $x$ of $H$. Roughly speaking, the mean index measures the sum of rotations of eigenvalues of $d(\varphi_H^t)_{x(t)}$ lying on the unit circle.
For exact details on the technical definition and the properties of the mean index, we refer the reader to \cite{Lo-ind,SZ}. A list of properties of the mean index can also be found in \cite{GG-gaps}. 

Up to sign, we define the Conley-Zehnder index as in \cite{Sa,SZ} and use the normalization such that for a non-degenerate minimum $x$ of an autonomous Hamiltonian with small Hessian we have $\mu_{\CZ}(x)=-n$, see \cite{GG-gaps}.
Together with the growth properties of the mean index, the most crucial property of the Conley-Zehnder index for our argument is the inequality $|\Delta_H(x)-\mu_{\CZ}(x)|\leq n$ for all periodic orbits $x$ and the fact that this inequality is strict, when $x$ is weakly non-degenerate.

The $k$th iteration of an orbit $x$ is denoted by $x^k$. The mean index and the action are both homogeneous with respect to iteration and satisfy the iteration formulas
\[
\mathcal{A}_{H^{(k)}}(x^k)=k\mathcal{A}_H(x) \text{ and } \Delta_{H^{(k)}}(x^k)=k\Delta_H(x).
\]

We define the Floer homology for a non-degenerate Hamiltonian $H$ as in \cite{Sa,SZ}. The homology is graded by the Conley-Zehnder index.
The Floer chain groups are generated by the contractible one-periodic orbits of $H$ and for the definition of the boundary operator, we consider solutions of the Floer equation

\begin{equation}
\frac{\partial u}{\partial s}+J_t(u)\frac{\partial u}{\partial t}=-\nabla H_t(u)
\label{Floer}
\end{equation}

\noindent with finite energy. 
As is well known, Floer trajectories $u$ for a non-degenerate Hamiltonian $H$ with finite energy $E(u)$ converge to periodic orbits $x$ and $y$ as $s$~goes to $\pm\infty$ and satisfy
\[
E(u):=\int_{-\infty}^{\infty}\int_{S^1}\left\|\frac{\partial u}{\partial s}\right\|^2\ dt\ ds=\mathcal{A}_H(x)-\mathcal{A}_H(y).
\] 
The Floer boundary operator counts Floer trajectories converging to periodic orbits $y$ and $x$ as $s\to\pm\infty$.

The conditions (H1) and (H2) on the Hamiltonian $H$ imply that solutions to the Floer equation \eqref{Floer} connecting two periodic orbits of $H$ are uniformly bounded in the $C^0$-norm, if the almost complex structure $J$ is sufficiently close to the standard almost complex structure on $M=T^\ast B$, see \cite{AS}. Thus the Floer homology is well-defined, since the one-dimensional the moduli spaces are compact by the same argument as in the case of a closed manifold.

 It is also proven in \cite{AS} that there exists an isomorphism
\begin{equation}
\HF_\ast(H)\cong H_{-\ast}(\Lambda_0B),
\label{Floeriso}
\end{equation}
where $\Lambda_0B$ is the space of contractible loops on $B$. (See also \cite{Se,SW,Vi} for similar results for somewhat different classes of Haminltonians.)
In particular, this implies $\HF_{-n}(H)\cong H_n(\Lambda_0 B)\neq 0$ for any Hamiltonian, which satisfies conditions (H1) and (H2).

This construction extends by continuity from non-degenerate Hamiltonians to all Hamiltonians satisfying conditions (H1) and (H2), see \cite{Sa}.

For two non-degenerate Hamiltonians $H^0$ and $H^1$, a homotopy from $H^0$ to $H^1$ induces a homomorphism of chain complexes which gives an isomorphism between the Floer homology groups $\HF_{*}(H^0)$ and $\HF_{*}(H^1)$ which is independent of the choice of homotopy. This map is defined analogously to the Floer boundary operator using the Floer equation \eqref{Floer} with the homotopy $H^s$ on the right hand side. Again it is proven in \cite{AS} that the trajectories connecting two one-periodic orbits are uniformly bounded and thus the one-dimensional moduli spaces are compact and the homotopy map is well-defined.

As the action decreases along Floer trajectories of a non-degenerate Hamiltonian $H$, we also have well-defined chain complexes only involving orbits with action in an interval $(a,\,b)$ for regular values $a$ and $b$ of the action functional $A_H$. The set of critical values of the action is called the action spectrum $\mathcal S(H)$. This complex gives rise to the \textit{filtered Floer homology} $\CF_{*}^{(a,\,b)}(H)$. The construction of filtered Floer homology also extends by continuity to degenerate Hamiltonians, since in this case the Floer homology is independent of the choice of a sufficiently small, non-degenerate perturbation.
Similarly to above, we get a homomorphism in the filtered homology induced by a homotopy between two non-degenerate Hamiltonians, if the homotopy is monotone decreasing at all point and for all times. This map is referred to as the \textit{monotone homotopy map} between filtered Floer homology groups.

By construction of the filtered Floer homology for non-degenerate Hamiltonians, we have a long exact sequence of filtered Floer homology groups
\begin{equation}
\label{eq:longexact}
\cdots\rightarrow \HF_{*}^{(a,\,b)}(K)\rightarrow \HF_{*}^{(a,\,c)}(K)\rightarrow \HF_{*}^{(b,\,c)}(K)\rightarrow \HF_{*-1}^{(a,\,b)}(K)\rightarrow\cdots
\end{equation}
for any non-degenerate Hamiltonian $K$ with $a,b,c\notin\mathcal S(K)$. The maps of this exact sequence commute with the monotone homotopy map.

\subsection{Local Floer homology}
\label{sec:locHF}

Let $x$ be an isolated one-periodic orbit of a Hamiltonian
$H\colon S^1\times M\to \R$. Pick a sufficiently small tubular
neighborhood $U$ of $x$ and consider a non-degenerate $C^2$-small
perturbation $\tilde{H}$ of $H$ supported in $U$.  Every (anti-gradient)
Floer trajectory $u$ connecting two one-periodic orbits of $\tilde{H}$ lying
in $U$ is also contained in $U$, provided that $\|\tilde{H}-H\|_{C^2}$ and
$supp(\tilde{H}-H)$ are small enough.  Thus, by the compactness and gluing
theorems, every broken anti-gradient trajectory connecting two such
orbits also lies entirely in $U$ by a similar argument as in Proposition \ref{hom decomp}, see also Remark \ref{rmk:locFloer}. The vector space (over $\Z_2$) generated by one-periodic
orbits of $\tilde{H}$ in $U$ is a complex with (Floer) differential defined
in the standard way. The continuation argument (see, e.g.,
\cite{SZ}) shows that the homology of this complex is independent
of the choice of $\tilde{H}$ and of the almost complex structure. We refer
to the resulting homology group $\HF_*(H,x)$ as the \emph{local
Floer homology} of $H$ at $x$.

\begin{example}
Assume that $x$ is a non-degenerate one-periodic orbit of $H$ with $\mu_{\CZ}(x)=k$.  Then
$\HF_l(H,x)=\Z_2$ when $l=k$ and $\HF_l(H,x)=0$ otherwise.
\end{example}

By definition, the \emph{support} of $\HF_*(H,x)$ is the collection
of integers $k$ such that $\HF_k(H,x)\neq 0$. Clearly, the group
$\HF_*(H,x)$ is finitely generated and hence supported in a finite
range of degrees, namely in $(\Delta_H(x)-n,\,\Delta_H(x)+n)$.

For a more detailed definition and a discussion of the properties of local Floer homology, see also \cite{Gi,GG-loc,GG-gaps}.

\section{Proof of  Theorem \ref{theorem}}
\label{sec:Beweis}
\subsection{Outline of the proof}
\label{sec:outline}

The key to proving Theorem \ref{theorem} is a geometrical description of \SDMins\ given in Proposition \ref{SDM geo}. In particular, we can assume the \SDMin\ to be a constant orbit $x_0$. Furthermore,  we can assume that $x_0$ is a strict local minimum of $H$ and that $H$ has arbitrarily small Hessian at $x_0$.

Then we use the squeezing method from \cite{BPS,Gi,GG-Rel} and construct Hamiltonians $H_+$ and $H_-$ such that $H_-<H<H_+$.
It suffices to show that a linear homotopy from $H_+$ to $H_-$ induces a non-zero map between the filtered Floer homology groups of $H_\pm$ for the action interval in question. This map factors through the filtered Floer homology group of $H$, which can therefore not be trivial.

As functions of the distance from $x_0$, the functions $H_+$ and $H_-$ are constructed similarly to the functions used in \cite{Gi,GG-gaps,He}; see Section \ref{sec:functions} for details.

For the Hamiltonians $H_\pm$ we use the direct sum decomposition from Proposition~\ref{hom decomp}, which has been established in \cite{He}. 
To prove that the monotone homotopy map is non-zero, it suffices to show that the restriction to one of the summands is an isomorphism.

The considered summand $\HF_*(H_\pm,\,U)$ for a neighborhood $U$ of the \SDMin\ $x_0$ depends only on the restriction of the functions $H_\pm$ to $U$ and the symplectic structure in $U$ and is independent of the ambient manifold. Thus we can view $U$ as an open set in any symplectic manifold of dimension $2n$ and the theorem follows as in the closed symplectically asherical case in \cite{Gi}.

\subsection{Geometric characterization of \SDMins}
\label{sec:sdmin}

In this section we state some geometric properties of \SDMins. The existence of a \SDMin\ enters the proof of Theorem \ref{theorem} only via those properties.

For the formulation of the geometric characterization of a \SDMin\ we first recall the definition of the norm of a tensor with respect to a coordinate system.
On a finite-dimensional vector space the norm $\left\|v\right\|_{\Xi}$ of a tensor $v$ with respect to a coordinate system $\Xi$ is by definition the norm of $v$ with respect to the inner product for which $\Xi$ is an orthonormal basis. For a coordinate system $\xi$ on a manifold $M$ near a point $x_0$, the natural coordinate basis in $T_{x_0}M$ is also denoted by~$\xi$.

\begin{prop}[\cite{GG-loc,GG-gaps}]
\label{SDM geo}
Let $x$ be a \SDMin\ of a Hamiltonian $H$ and let $x_0=x(0)\in M$. Then there exists a sequence of contractible loops $\eta_i$ of Hamiltonian diffeomorphisms such that $x(t)=\eta_i^t(x_0)$, i.e each loop $\eta_i$ sends $x_0$ to $x$.
Furthermore, the Hamiltonians $K^i$ given by $\varphi_H^t=\eta_i\circ\varphi_{K^i}^t$ and the loops $\eta_i$ satisfy the following conditions:
\begin{itemize}
\item [\textbf{(K1)}] the point $x_0$ is a strict local minimum of $K_t^i$ for $t\in S^1$;
\item [\textbf{(K2)}] there exist symplectic bases $\Xi^i$ of $T_{x_0}M$ such that 
\[
\left\|d^2(K_t^i)_{x_0}\right\|_{\Xi^i}\rightarrow 0 \text{ uniformly in } t\in S^1;
\]
\item [\textbf{(K3)}] the loop $\eta_i^{-1}\circ \eta_j$ has identity linearization at $x_0$ for all $i$ and $j$ (i.e. for all $t\in S^1$ we have $d\big((\eta_i^t)^{-1}\circ \eta_j^t\big)_{x_0}=I$), and is contractible to $id$ in the class of such loops.
\end{itemize}
\end{prop}

A proof of the analogous proposition for \SDMaxs\ can be found in \cite{GG-loc,GG-gaps}. Here we are only going to show how this proposition follows from the case of a \SDMax.

\begin{definition}
\label{SDMax}
An isolated $k$-periodic orbit $x$ of a $k$-periodic Hamiltonian $H$ is called a
 \textit{\SDMax} of $H$ if \[\Delta_H(x)=0 \text{ and } \HF_n(H,x)\neq 0.\]
\end{definition}

When the concept of a \SDMax\ was introduced in~\cite{Hi} by Hingston and the first formal definition given in \cite{Gi}, the geometric characterization was used as a definition of \SDMaxs. It is shown in \cite{GG-loc} that this characterization is equivalent to the definition and also that (K1) and (K2) already imply (K3) as a formal consequence. 

\begin{proof}[Proof of Proposition \ref{SDM geo}]
The (local) Floer homology of $H^{inv}=-H_t\circ\varphi_H^{t}$ generating the inverse flow of $\varphi_H^t$ can be calculated from the Floer homology of $H$, since all one-periodic orbits of $H$ give rise to one-periodic orbits of $H^{inv}$ by reversing the orientation. The Conley-Zehnder index and the action of a periodic orbit of $H^{inv}$ are the negatives of the Conley-Zehnder index and the action of the corresponding periodic orbit of $H$; and the Floer trajectories are the Floer trajectories of $H$ with $s$ replaced by $-s$ and orientation in the $t$-direction reversed.

Let $y(t)=x(-t)$ be the \SDMin\ traversed in opposite direction. Then $y$ is a 1-periodic orbit of the Hamiltonian $H_t^{inv}$.
The properties of the \SDMin\ $x$ imply that $\Delta_{H^{inv}}(y)=-\Delta_H(x)=0$ and $\HF_n(H^{inv},y)=\HF_{-n}(H,x)\neq 0$. Therefore, $y$ is a \SDMax\ of $H^{inv}$ and we can use the geometric characterization of \SDMaxs\ from \cite{GG-loc,GG-gaps,He} to construct Hamiltonians $G_t^i$ and loops $\gamma_i^t$ such that 

\begin{enumerate}
\item the point $x_0$ is a strict local maximum of $G_t^i$ for $t\in S^1$,
\item there exist symplectic bases $\Xi^i$ of $T_{x_0}M$ such that 
\[
\left\|d^2(G_t^i)_{x_0}\right\|_{\Xi^i}\rightarrow 0 \text{ uniformly in } t\in S^1,
\]
\item the loop $\gamma_i^{-1}\circ \gamma_j$ has identity linearization at $x_0$ for all $i$ and $j$ and is contractible to $id$ in the class of such loops,
\item $\varphi_H^{-t}=\varphi_{G^i}^t\circ\gamma_i^t$.
\end{enumerate}

Now we define the loops $\eta_i$ and the Hamiltonians $K^i$ by inverting the loops $\gamma_i$ and the flows of the Hamiltonians $G^i$, i.e., $\eta_i^t=\gamma_i^{-t}$ and $K_t^i=-G_t^i\circ\varphi_{G^i}^{t}$. Then we have $\varphi_H^t=\eta_i\circ\varphi_{K^i}^t$, as required in the proposition.

The properties (K1), (K2) and (K3) of the loops $\eta_i$ and the Hamiltonians $K^i$ follow directly from the properties (i), (ii) and (iii) of $\gamma_i$ and $G^i$ with the same coordinate systems $\Xi^i$ of $T_{x_0}M$.
\end{proof}

\begin{rmk}
The equation $\varphi_H^{-t}=\varphi_{G^i}^t\circ\gamma_i^t$ in (iv) is a modification of the geometric characterization of \SDMaxs\ in \cite{GG-loc,GG-gaps,He}. In those papers and in the definition of a \SDMax\ in \cite{Gi}, the requirement takes the form $\varphi_H^{-t}=\gamma_i^t\circ\varphi_{G^i}^t$. But in the construction of the loops and the Hamiltonians, the order of composition is not crucial, see \cite{Gi} for details.
In the proof of Theorem \ref{theorem}, we need the order of composition to be $\eta_i\circ\varphi_{K^i}^t$ to ensure that a composition $\eta_i^t\circ\varphi_F^t$ for an autonomous Hamiltonian $F$ is generated by a one-periodic Hamiltonian. This would not necessarily be the case if the order of composition is changed.
\end{rmk}

\begin{rmk}
The loops $\eta_i^{-1}\circ \eta_j$ are loops of Hamiltonian diffeomorphisms fixing $x_0$.
The construction of the loops $\gamma_i$ in \cite{Gi} for the case of a \SDMax\ shows that the loops $\gamma_i$ can be chosen such that $\gamma_i^{-1}\circ \gamma_j$ are supported in an arbitrarily small neighborhood of $x_0$. Hence also the loops $\eta_i=\gamma_i^{-1}$ can be chosen to be supported near $x_0$.
\label{supp eta}
\end{rmk}

\subsection{The functions $H_+$ and $H_-$}
\label{sec:functions}

By Proposition \ref{SDM geo} above, it suffices to prove the theorem for the function $K^1$ and the constant orbit $x_0$ as \SDMin. We keep the notation $H$ for $K^1$. Fix a neighborhood $W$ of $x_0$ such that $x_0$ is a strict global minimum of $H$ on $W$ and that there exist Darboux coordinates for $M$ in $W$. Furthermore, choose $W$ such that $\|p\|\leq C$ in $W$ for some possibly very large constant $C>0$. 
We also fix now an almost complex structure $J$ on $M$ that is compatible with $\omega$.

Let $U$ and $V$ be balls centered at $x_0$ and contained in $W$. 
We then construct the function $H_-$ and an auxiliary function $F$ to be of the form shown in Figure \ref{fig:functions}.

More concretely, we fix balls 
\[
B_{r_-}\subset B_{r_+}\subset B_r\subset U\subset V\subset B_R\subset B_{R_-}\subset B_{R_+}\Subset W
\]
centered at $x_0$.
In $W$, the function $H_-$ takes the following form as a function of the distance from $x_0$:
\begin{itemize}
\item $H_-\geq H$ and $H_-\equiv c=H(x_0)$ on $B_{r_-}$;
\item on $B_{r_+}\setminus B_{r_-}$ the function $H_-$ is monotone increasing;
\item on $B_r\setminus B_{r_+}$ the function is constant;
\item in the shell $B_R\setminus B_r$ the function is monotone decreasing, linear as a function of the square of the distance from $p$ with small slope $\alpha$ on $V\setminus U$ such that there are no one-periodic orbits in $V\setminus B_r$;
\item the function $H_-$ is again constant on $B_{R_-}\setminus B_R$ with a value less than $c$;
\item it is monotone decreasing on $B_{R_+}\setminus B_{R_-}$;
\item outside $B_{R_+}$, the function $H_-$ is constant and equal to its minimum.
\end{itemize}

\begin{figure}
\label{fig:functions}
\centering
\scalebox{0.4}[0.3]{
   \includegraphics{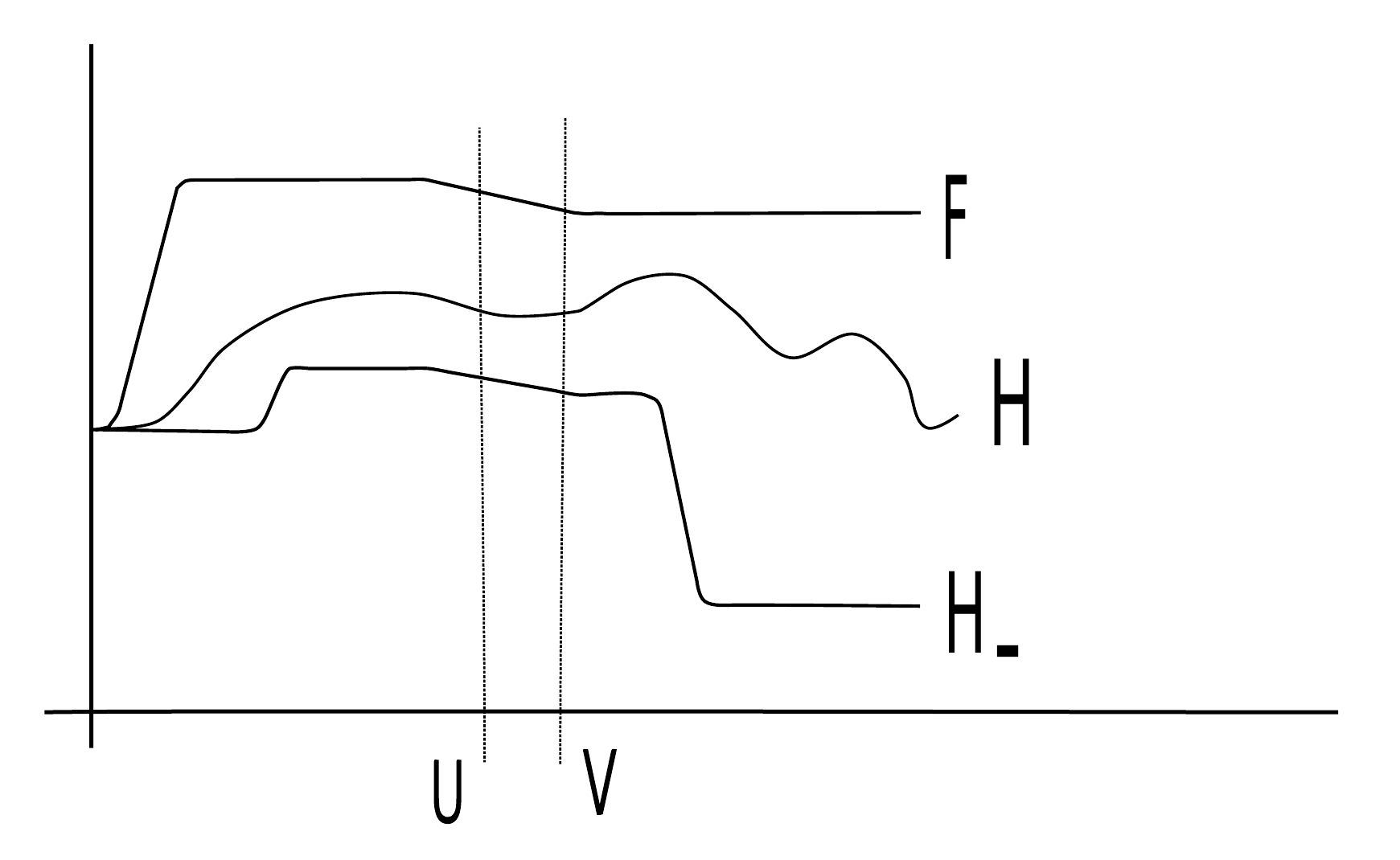}}
\caption{The functions $H_-$ and $F$ as functions of the distance from $x_0$.}
\end{figure}   

The function $H_-$ also has to meet the requirements (H1) and (H2). This can be realized by defining $H_-$ to be a positive, non-degenerate quadratic form for $\|p\|\geq C$ i.e. we choose $H_-$ to satisfy \eqref{quadgrow} and thus conditions (H1) and (H2). The coefficients of the quadratic form are chosen such that we have $H_-\leq H$ on $M$. Thus $H_-$ is of the general form mentioned in Example~\ref{classical ex}.

This function is constructed very similar to the ones used in \cite{Gi,GG-gaps,He} near a \SDMax. More concretely, up to an additive constant and the quadratic growth condition, this is the negative of the function $H_+$ used in \cite{He}, since we are using a \SDMin\ here.
See also Section \ref{sec:HF} for details about the choices made in the construction of $H_-$.

Let us now turn to the construction of $H_+$. We use the existence of a \SDMin. The geometrical characterization of \SDMins\ in Proposition \ref{SDM geo} and Remark \ref{supp eta} imply that we have
\begin{itemize}
\item a loop $\eta^t$ of Hamiltonian diffeomorphisms fixing $x_0$, which is supported in~$U$ and
\item a system of coordinates $\xi$ on a neighborhood $W$ of $x_0$
\end{itemize}
such that the Hamiltonian $K$ generating the flow $\eta^{-t}\circ\varphi_H^t$ has a strict local minimum at $x_0$ and $\max_t\left\|d^2(K_t)_{x_0}\right\|_{\xi}$ is sufficiently small. The loop $\eta$ is contractible in the class of loops having identity linearization at $x_0$.
Let $G_s^t$ be a Hamiltonian generating such a homotopy $\eta_s^t$ normalized by $G_s^t(x_0)\equiv 0$. We then normalize $K$ by the additional requirement that $K_t(x_0)\equiv c$ (or equivalently that $H=K\# G^0$).

Then there exists a function $F$, depending on the coordinate system $\xi$, such that
\begin{itemize}
\item $\left\|d^2F_{x_0}\right\|_{\xi}$ is sufficiently small,
\item $F\geq K$ and $F(x_0)=c=H(x_0)$ is the global minimum of $F$.
\end{itemize}
To be more precise, in $B_r$ the function $F$ is the negative of a bump function centered at $x_0$. Furthermore, $F$ is chosen to differ from $H_-$ only by a constant in $V\setminus U$. The last condition is needed to have the direct sum decomposition in Proposition~\ref{hom decomp} below for the filtered Floer homology groups of $H_-$ and $F$ and to ensure that this decomposition is compatible with the monotone homotopy map for a linear homotopy from $H_-$ to $F$.
Outside $ B_R$ and for $\|p\|\leq C$, we choose $F$ to be constant. As for $H_-$, we define $F$ to be of the form from Example~\ref{classical ex}, i.e., a positive, non-degenerate quadratic form for $\|p\|>C$, to ensure that $F$ satisfies the growth conditions (H1) and (H2). The coefficients of the quadratic form are fixed such that all functions of the homotopy $F^s=G^s\# F$ satisfies $F^s\geq H_-$ and $F^1=F\geq K$.

This is an isospectral homotopy, i.e., a homotopy such that the action spectrum $\mathcal{S}(F^s)$ is independent of $s$.
We define the function $H_+$ by
\[
H_+:=G^0\# F\geq G^0\# K=H.
\]
Since $\eta$ is supported in $U$, the function $G^0$ is constant outside $U$ and $H_+$ differs from $F$ and $H_-$ only by the constant value of $G^0$ on $\bar V\setminus U$. Therefore we also have the direct sum decomposition from Proposition \ref{hom decomp} for $H_+$. It is compatible with the homomorphism induced by the homotopy $F^s$ and the monotone homotopy map for a homotopy from $H_+$ to $H_-$.

\subsection{The Floer homology of $H_{\pm}$ and the monotone homotopy map}
\label{sec:HF}

Consider the splitting of Floer chain groups into the direct sum
\begin{equation}
\CF_{*}^{(a,\,b)}(K)=\CF_{*}^{(a,\,b)}(K,U)\oplus \CF_{*}^{(a,\,b)}(K;M,U)
\label{chain sum}
\end{equation}
where the first summand is generated by the one-periodic orbits in $U$. The second summand is spanned by all the remaining orbits.

\begin{prop}
\label{hom decomp}
Let $K$ be a Hamiltonian on $M$ satisfying conditions (H1) and (H2) and let $U$ and $V$ be two open sets such that $U\subset V$. Assume that both $U$ and $V$ are bounded by level sets of $K$ and that $K$ is autonomous on $V\setminus U$ with not one-periodic orbits in this domain.
Then there exists an $\epsilon>0$, depending only on $J$, the open sets $U$ and $V$ and on $K|_{V\setminus U}$ such that \eqref{chain sum} gives rise to a direct sum decomposition of homology
\begin{equation}
\HF_{*}^{(a,\,b)}(K)=\HF_{*}^{(a,\,b)}(K,U)\oplus \HF_{*}^{(a,\,b)}(K;M,U)
\label{Homology sum}
\end{equation}
whenever the action interval $(a,\,b)$ is chosen such that $b-a<\epsilon$.
\end{prop}

This proposition is proven in \cite{He}, using an energy bound along the lines of \cite{Us}, in the situation of an arbitrary symplectic manifold.
In the case of a symplectically rational manifold, this direct sum decomposition was proven in \cite{GG-gaps} if $K$ is constant on $V\setminus U$. This proof relies on energy bounds for $J$-holomorphic curves.
Here we are going to apply Proposition~\ref{hom decomp} to the functions $H_\pm$ and $F$. As the cotangent bundle is symplectically aspherical, we could also have taken those functions to be constant on $V\setminus U$ and used the direct sum decomposition from \cite{GG-gaps}.

\begin{rmk}
\label{rmk:locFloer}
This direct sum decomposition generalizes the concept of the local Floer homology. Indeed, if $x$ is an isolated periodic orbit with action $c$, one can find neighborhoods $U$ and $V$ of $x$ containing no other periodic orbits. Then for a sufficiently small $\eps>0$, the summand $\HF_{*}^{(c-\eps,\,c+\eps)}(K,U)$ is exactly the local Floer homology $\HF_\ast(K,x)$ of the Hamiltonian $K$ at $x$. Indeed, for sufficiently small non-degenerate perturbations $\tilde{K}$ of $K$, the periodic orbits of $\tilde{K}$ near $x$ have action in the interval $(c-\eps,\,c+\eps)$ and therefore are in this summand of the Floer homology.
\end{rmk}

In the construction of $H_-$, the constants are chosen such that this proposition can be applied.
To be more precise, we first choose some small constant $\alpha_0>0$ such that $\alpha_0/\pi$ is irrational. Then we fix the Hamiltonian $H_-$ on $B_r$ and pick $\eps>0$ smaller than the energy bound from Proposition~\ref{hom decomp} for a Hamiltonian linear with slope $\alpha_0$ on $V\setminus U$.
Using these choices, we take a sufficiently large order of iteration $k$ as in \cite{Gi,GG-gaps}. 
Furthermore, we now fix $H_-$ outside $B_r$ with slope $\alpha=\alpha_0/k$ on $V\setminus U$.
We thus have the direct sum decomposition of filtered Floer homology by Proposition~\ref{hom decomp} for $H_-^{(k)}$. At this point we choose some $\delta_k\in(0,\,\eps/2)$, depending on $k$, to ensure that the action intervals $(kc+\delta_k,\,kc+\eps)$ and $(kc-\delta_k,\,kc+\delta_k)$ are sufficiently small for the direct sum decomposition \eqref{Homology sum}.

Thus the proposition applies to the functions $H_\pm$ and $F$ and we can restrict ourselves to the summand of the Floer homology containing the orbits in $U$. Since $U$ is contained in a Darboux neighborhood of $x_0$, this summand does only depend on the restriction of the function to $U$ and is independent of the ambient manifold.

For the calculation of this part of the Floer homology groups, we refer the reader to \cite{Gi,GG-loc}.
The functions $H_-$ and $F$ are autonomous and generate the inverses of the flows of the Hamiltonians $H_+$ and $F$ used in \cite{Gi,He}, hence they are just the negatives of those functions. The Floer equation \eqref{Floer} is the same equation as for the functions in \cite{Gi} with the orientation in both coordinates of $S^1\times\R$ reversed and the negative Hamiltonians. Thus the numbers of connecting Floer trajectories are equal.

Similarly, the homotopy from $F$ to $H_-$ is the negative of the homotopy  from $H_+$ to $F$ in \cite{Gi} and we obtain an isomorphism of Floer homology goups
\[
\mathbbm{Z}_2\cong \text{HF}_{-n-1}^{(kc-\eps,\,kc-\delta_k)}(F^{(k)},U)\to \text{HF}_{-n-1}^{(kc-\eps,\,kc+\delta_k)}(H_-^{(k)},U)\cong\mathbbm{Z}_2
\]
by the same argument.

Then we have the commutative diagram
\begin{displaymath}
\xymatrix{
\HF_{-n-1}^{(kc-\eps,\,kc-\delta_k)}(F^{(k)},U)\ar[r]^{\cong}\ar[rd]_{\cong}&\HF_{-n-1}^{(kc-\eps,\,kc-\delta_k)}(H_+^{(k)},U)\ar[d]^{\Psi}\\
&\HF_{-n-1}^{(kc-\eps,\,kc-\delta_k)}(H_-^{(k)},U)
}
\end{displaymath}
where  the horizontal map is induced by the isopsectral homotopy $F^s$ and the other maps are monotone homotopy maps. As in the case of a closed symplectically aspherical manifold, the isospectral homotopy $F^s$ induces an isomorphism in this summand of the filtered Floer homology.
The commutativity is established in the same way as in the case of a closed symplectically aspherical manifold. 
The diagonal map is an isomorphism by the same argument as in \cite{Gi} using the long exact sequence \eqref{eq:longexact} of filtered Floer homology to go over to the action interval $(kc-\delta_k,kc+\delta_k)$.
By the commutativity of this diagram, the map $\Psi$ is also an isomorphism. Thus the monotone homotopy map
\[
\HF_{-n-1}^{(kc-\eps,\,kc-\delta_k)}(H_+^{(k)})\to \HF_{-n-1}^{(kc-\eps,\,kc-\delta_k)}(H_-^{(k)})
\]
 is non-zero and this map factors through the Floer homology group of $H$, which we want to show to be non-trivial. This proves Theorem~\ref{theorem}.


\begin{thebibliography}{XXXX}

\bibitem[AS]{AS} A. Abbondandolo, M. Schwarz,
On the Floer homology of cotangent bundles,
\emph{Comm. Pure Appl. Math.}, \textbf{59} (2006), 254-316.

\bibitem[BPS]{BPS}P. Biran, L. Polterovich, D. Salamon,
Propagation in Hamiltonian dynamics and relative symplectic homology,
\emph{Duke Math. J.}, \textbf{119} (2003), 65-118.

\bibitem[Co]{Co}C.C. Conley, 
Lecture at the University of Wisconsin, April 6, 1984.

\bibitem[CZ]{CZ}C.C. Conley, E. Zehnder, 
Morse-type index theory for flows and periodic solutions for 
Hamiltonian equations, 
\emph{Comm. Pure Appl. Math.} \textbf{37} (1984), 207-253. 
 
\bibitem[FH]{FrHa}J. Franks, M. Handel, 
Periodic points of Hamiltonian surface diffeomorphisms, 
\emph{Geom. Topol.}, \textbf{7} (2003), 713-756.

\bibitem[Gi1]{Gi-magn}V.L. Ginzburg,
On closed trajectories of a charge in a magnetic field. An application of symplectic geometry, 
in \emph{Contact and Symplectic Geometry}, Eds.: C.B. Thomas, INI Publications, Cambridge University Press, Cambridge (1996), 131-148.

\bibitem[Gi2]{Gi}V.L. Ginzburg, 
The Conley Conjecture, 
\emph{Ann. of Math.}, \textbf{172} (2010), 1127-1180.

\bibitem[GG1]{GG-Rel}V.L. Ginzburg, B. G\"urel, 
Relative Hofer-Zehnder capacity and periodic orbits in twisted cotangent bundles, 
\emph{Duke math. J.}, \textbf{123} (2004), 1-47.

\bibitem[GG2]{GG-gaps}V.L. Ginzburg, B.Z. G\"urel, 
Action and index spectra and periodic orbits in Hamiltonian dynamics, 
\emph{Geom. Topol.}, \textbf{13} (2009), 2745-2805.

\bibitem[GG3]{GG-loc}V.L. Ginzburg, B.Z. G\"urel, 
Local Floer homology and the action gap, 
\emph{J. Sympl. Geom.}, \textbf{8} (2010), 323-357.

\bibitem[He]{He} D. Hein,
The Conley conjecture for irrational symplectic manifolds,
Preprint 2009, arXiv:0912.2064, to appear in \emph{J. Sympl. Geom.}

\bibitem[Hi]{Hi}N. Hingston, 
Subharmonic solutions of Hamiltonian equations on tori, 
\emph{Ann. of Math.}, \textbf{170} (2009), 529-560. 

\bibitem[Lo1]{Lo-CC} Y. Long,
Multiple periodic points of the Poincar\'e map of Lagrangian systems on tori,
\emph{Math. Z.}, \textbf{233} (2000), 443-470.

\bibitem[Lo2]{Lo-ind} Y. Long,
Index Theory for Symplectic Path with Applications,
\emph{Progress in Math.}, \textbf{207} (2002). Birkh\"auser Verlag, Basel.

\bibitem[Lu]{Lu} G. Lu,
The Conley conjecture for Hamiltonian systems on the cotangent bundle and its analogue for Lagrangian systems,
\emph{J. Funct. Anal.}, \textbf{256} (2009), 2967-3034.

\bibitem[Ma]{Ma} M. Mazzucchelli,
The Lagrangian Conley Conjecture,
Preprint 2008, arXiv:0810.2108, to appear in \emph{Comment. Math. Helv.}

\bibitem[Sa]{Sa}D.A. Salamon, 
Lectures on Floer homology, 
in \emph{Symplectic Geometry and Topology}, Eds.: Y. Eliashberg and L. Traynor, 
IAS/Park City Mathematics series, \textbf{7} (1999), pp. 143-230.

\bibitem[SW]{SW}D. Salamon, J. Weber,
Floer homology and the heat flow,
\emph{Geom. Funct. anal.} \textbf{16} (2006), 1050-1138.


\bibitem[SZ]{SZ}D. Salamon, E. Zehnder, 
Morse theory for periodic solutions of Hamiltonian systems and the Maslov index, 
\emph{Comm. Pure Appl. Math.}, \textbf{45} (1992), 1303-1360.

\bibitem[Se]{Se} P. Seidel,
A biased view of symplectic cohomology,
Lecture notes from \emph{Current Developments in Mathematics}, 2006,
211--253, Int. Press, Somerville, MA, 2008.

\bibitem[Us]{Us}M. Usher, 
Floer homology in disk bundles and symplectically twisted geodesic flows, 
\emph{J.~Mod. Dyn.}, \textbf{3} (2009), 61-101.

\bibitem[Vi]{Vi}C. Viterbo,
Functors and Computations in Floer homology and applications, Part II,
Preprint 1996 (revised 2003), available from 
http://www.math.polytechnique.fr/cmat/viterbo/Prepublications.html. 

\end{thebibliography}
\end{document}